\documentclass[11pt]{amsart}

\usepackage{amsmath, amsthm, amssymb}

\usepackage{lmodern}
\usepackage{palatino}                       
\usepackage[bigdelims,vvarbb]{newpxmath}    
\usepackage[scaled=0.95]{inconsolata}       
\usepackage[T1]{fontenc}                    

\usepackage{comment}

\usepackage[abs]{overpic}		  

\usepackage{xcolor}
\definecolor{indigo}{rgb}{0.29, 0.0, 0.51}  
\usepackage[colorlinks, urlcolor=indigo, linkcolor=indigo, citecolor=indigo]{hyperref}

\usepackage[hcentering, top = 1.5in, total={5.9in, 8.2in}]{geometry}  

\theoremstyle{plain}
\newtheorem{theorem}{Theorem}
\newtheorem{corollary}[theorem]{Corollary}
\newtheorem{proposition}[theorem]{Proposition}
\newtheorem{lemma}[theorem]{Lemma}

\newtheorem{conjecture}[theorem]{Conjecture}

\theoremstyle{definition}

\theoremstyle{remark}
\newtheorem{remark}[theorem]{Remark}
\newtheorem{example}[theorem]{Example}

\numberwithin{theorem}{section}

\newcommand{\dfn}[1]{{\em #1}}        
\newcommand{\R}{\mathbb{R}}           
\newcommand{\Q}{\mathbb{Q}}           
\newcommand{\modp}[1]{\;(\!\!\!\!\!\!\mod #1)}
\newcommand{\case}[1]{\begin{cases} #1 \end{cases}} 

\makeatletter
\newcommand*\bigcdot{\mathpalette\bigcdot@{0.6}}
\newcommand*\bigcdot@[2]{\mathbin{\vcenter{\hbox{\scalebox{#2}{$\m@th#1\bullet$}}}}}
\makeatother

\DeclareMathOperator\tb{tb}                   

\DeclareFontFamily{U} {cmr}{}
\DeclareFontShape{U}{cmr}{m}{n}{
  <-6> cmr5
  <6-7> cmr6
  <7-8> cmr7
  <8-9> cmr8
  <9-10> cmr9
  <10-12> cmr8
  <12-> cmr9}{}
\DeclareSymbolFont{Xcmr} {U} {cmr}{m}{n}
\DeclareMathSymbol{\Phi}{\mathord}{Xcmr}{8}

\begin{document}

\title[Strongly fillable contact structures without Liouville fillings]{Strongly fillable contact structures\\without Liouville fillings} 

\author{Hyunki Min}
\address{Department of Mathematics \\ Massachusetts Institute of Technology \\ Cambridge, MA}
\email{hkmin@mit.edu}


\begin{abstract}
  We introduce a new method to obstruct Liouville and weak fillability. Using this, we show that various rational homology 3--spheres admit strongly fillable contact structures without Liouville fillings, which extends the result of Ghiggini on a family of Brieskorn spheres. We also make partial progress on a conjecture of Ghiggini and Van-Horn-Morris. 
\end{abstract}

\maketitle

\section{Introduction}

One of the fundamental problems in contact geometry is to determine which contact structure bounds a symplectic manifold with a certain boundary condition, which is called a \dfn{symplectic filling}. There is a hierarchy among contact structures according to the type of symplectic fillings they admit: 
\[
  \{\text{Stein fillable}\} \subsetneq \{\text{Liouville fillable}\} \subsetneq \{\text{Strongly fillable}\} \subsetneq \{\text{Weakly fillable}\} \subsetneq \{\text{Tight}\}.
\]
It is known that all inclusions are proper in every dimension. The properness of each inclusion in $3$-dimension was shown by Bowden \cite{Bowden:noStein}, Ghiggini \cite{Ghiggini:nonfillable}, Eliashberg \cite{Eliashberg:nostrong}, and Etnyre and Honda \cite{EH:nofilling}, respectively. Also, the properness of each inclusion in higher dimensions was shown by Bowden, Crowley and Stipsicz \cite{BCS:nostein}, Zhou \cite{Zhou:nonfillable-highdim,Zhou:nonfillable2} (see also Ghiggini and Niederkr\"uger \cite{GN:nonfillable-highdim}), Bowden, Gironella and Moreno \cite{BGM:nostrong}, and Massot, Niederkr\"uger and Wendl \cite{MNW:nofilling}, respectively.

In this paper, we focus on the second and fourth inclusions in $3$-dimension, which are still not very well understood. For the second inclusion, the only known results so far are by Ghiggini \cite{Ghiggini:nonfillable} on a family of Brieskorn spheres $-\Sigma(2,3,6n-1)$ for $n\geq3$ and Tosun \cite{Tosun:classification} on $-\Sigma(2,3,6n+1)$ for $n\geq2$ (by applying the same argument of Ghiggini). Ghiggini used contact invariants in Heegaard Floer homology, so the nature of the obstruction is gauge theoretic, while most filling obstructions come from the holomorphic curve theory.  

In Theorem~\ref{thm:mixed} and Proposition~\ref{prop:mixed}, we introduce a new method to obstruct Liouville fillability. We prove that a large class of Legendrian knots, called \dfn{mixed Legendrian knots}, preserve Liouville non-fillability under contact surgery. An immediate application is that these Legendrian knots in a contact manifold containing planar torsion obstruct Liouville fillability under contact surgery, see Corollary~\ref{cor:planar-torsion}. To prove Theorem~\ref{thm:mixed}, we use the result of Menke \cite{Menke:decomposition} that we can decompose a Liouville domain along a properly embedded solid torus bounded by a \dfn{mixed torus}. Since Menke used holomorphic curves to prove his result, the nature of this obstruction is from the holomorphic curve theory. 

In Theorem~\ref{thm:non-Liouville}, we construct various contact $3$--manifolds that admit strong fillings but do not admit Liouville fillings using Theorem~\ref{thm:mixed}. This can be considered as a generalization of the results of Ghiggini and Tosun. This is because $-\Sigma(2,3,6n\pm1)$ are obtained by $1/n$--surgery on the trefoils, while the manifolds constructed in Theorem~\ref{thm:non-Liouville} are obtained by surgeries on any genus one fibered knot in any rational homology sphere. One interesting aspect is that the surgery coefficients depend on the fractional Dehn twist coefficient of the monodromy of a given knot. 

For the fourth inclusion, there have been several results on Seifert fibered spaces, see \cite{EH:nofilling,GLS:noweak, Matkovic:weak, PVHM:weak} for examples. In Theorem~\ref{thm:mixed}, we prove that a mixed Legendrian knot also preserves weak non-fillability under contact surgery. Using this, we show that various $3$--manifolds admit a tight contact structure without weak fillings, which were not accessible previously. See the examples in Section~\ref{sec:non-weak}.  


In Section~\ref{sec:Brieskorn}, we make partial progress on a conjecture of Ghiggini and Van-Horn-Morris \cite[Conjecture~1.2]{GVHM:classification}. In Theorem~\ref{thm:triangle}, we show that most of the tight contact structures on $-\Sigma(2,3,6n\pm1)$ are not Liouville fillable, while in the forthcoming paper with Etnyre and Tosun \cite{EMT:torus}, we disprove the conjecture by finding Stein fillings of some tight contact structures on $-\Sigma(2,3,23)$.

\subsection{Contact structures without Liouville fillings} \label{sec:non-liouville}
Let $L$ be a Legendrian knot. We say $L$ is a \dfn{mixed Legendrian knot} if $L = S_+S_-(L')$ for some Legendrian knot $L'$. In other words, $L$ is a doubly stabilized Legendrian knot with different signs. Now we state our first main theorem that contact surgery on a mixed Legendrian knot preserves Liouville and weak non-fillability. 

\begin{theorem} \label{thm:mixed}
  Suppose $L$ is a mixed Legendrian knot in a closed contact $3$--manifold $(Y,\xi)$. Let $(Y_{(r)}(L),\xi_{(r)})$ be a result of contact $(r)$--surgery on $L$. If $\xi$ is not Liouville (resp.~weakly) fillable, then $\xi_{(r)}$ is also not Liouville (resp.~weakly) fillable for any $r \in \mathbb{Q}$.  
\end{theorem}

\begin{remark} \label{rmk:overtwisted}
  In fact, any non-negative contact surgery on a mixed Legendrian knot produces an overtwisted contact structure. See the proof of Theorem~\ref{thm:mixed}.
\end{remark}
 
See Proposition~\ref{prop:mixed} for a simple version of the theorem. Also notice that if $L$ is null-homologous, then contact $(r)$--surgery is topological $(\tb(L)+r)$--surgery.

In \cite{Wendl:hierarchy}, Wendl defined \dfn{planar torsion}, which is an obstruction for strong fillability. A direct application of Theorem~\ref{thm:mixed} is that mixed Legendrian knots in a contact $3$--manifold containing planar torsion obstruct Liouville fillability under contact surgery. 

\begin{corollary} \label{cor:planar-torsion}
  Suppose $L$ is a mixed Legendrian knot in a closed contact $3$--manifold $(Y,\xi)$ containing planar $k$--torsion for $k \geq 0$. Let $(Y_{(r)}(L),\xi_{(r)})$ be a result of contact $(r)$--surgery on $L$. Then $\xi_{(r)}$ is not Liouville fillable for any $r \in \mathbb{Q}$.
\end{corollary}

Using Theorem~\ref{thm:mixed}, we can find various homology $3$--spheres which admit a strongly fillable contact structure without Liouville fillings. For any $s \in \Q \cup \{\infty\}$, we define a set of (extended) rational numbers $\mathcal{R}(s)$ as follows: 
\[
  \mathcal{R}(s) := \case{(0,s) & s > 0, \\ (0,\infty) & s = \infty, \\ (0,\infty] \cup (-\infty, s) & s < 0.} 
\]
Let $K$ be a fibered knot in a closed $3$--manifold $Y$. Suppose $\phi$ is the monodromy of $K$ and $c(\phi)$ is the fractional Dehn twist coefficient of $\phi$. Then we define
\[
  n_K := \begin{cases} 3 + \lceil c(\phi) \rceil &\text{if } \phi \text{ is pseudo-Anosov,}\\ 4 + \lfloor c(\phi) \rfloor &\text{otherwise.}\end{cases}
\]
Now we are ready to state the second main theorem.

\begin{theorem} \label{thm:non-Liouville}
  Let $Y$ be a rational homology $3$--sphere and $K$ be a genus one fibered knot in $Y$. Then $Y_r(K)$ admits a strongly fillable contact structure which is not Liouville fillable for any $r \in \mathcal{R}(1/n_K)$.
\end{theorem}


We can apply the theorem to the knots in $S^3$ and find infinite families of strongly fillable contact $3$--manifolds without Liouville fillings.

\begin{example} \label{ex:trefoil}
  Let $T_{2,3}$ be the right-handed trefoil. Ghiggini \cite{Ghiggini:nonfillable} showed $S^3_{1/n}(T_{2,3})$ admits a strongly fillable contact structure without Liouville fillings for any integer $n \geq 3$. Applying Theorem~\ref{thm:non-Liouville}, we can show $S^3_r(T_{2,3})$ admits a strongly fillable contact structure without Liouville fillings for any rational number $0<r<1/4$ (notice that $\phi$ is periodic and $c(\phi) = 1/6$). 
\end{example}

\begin{example}
  Let $T_{2,-3}$ be the left-handed trefoil. Tosun \cite{Tosun:classification} showed that $S^3_{1/n}(T_{2,-3})$ admits a strongly fillable contact structure without Liouville fillings for any integer $n\geq2$. Applying Theorem~\ref{thm:non-Liouville}, we can show $S^3_r(T_{2,-3})$ admits a strongly fillable contact structure without Liouville fillings for any rational number $0<r<1/3$ (notice that $\phi$ is periodic and $c(\phi) = -1/6$). 
\end{example}

\begin{example} \label{ex:f8}
  Let $K$ be the figure-eight knot. Applying Theorem~\ref{thm:non-Liouville}, we can show an infinite family of hyperbolic homology spheres $S^3_r(K)$ admits a strongly fillable contact structure without Liouville fillings for any rational number $0 < r < 1/3$ (notice that $\phi$ is pseudo-Anosov and $c(\phi)=0$).
\end{example}

For general closed $3$--manifolds, we can still find weakly fillable contact structures without Liouville fillings, but they are not strongly fillable in general. 

\begin{theorem} \label{thm:weak-non-Liouville}
  Let $Y$ be a closed oriented $3$--manifold and $K$ be a genus one fibered knot in $Y$. Then $Y_r(K)$ admits a weakly fillable contact structure which is not Liouville fillable for any $r \in \mathcal{R}(1/n_K)$.
\end{theorem}

The main idea of the proof of Theorem~\ref{thm:mixed} is following. Menke \cite{Menke:decomposition} showed that we can decompose a Liouville filling along a properly embedded solid torus bounded by a \dfn{mixed torus} (See Section~\ref{sec:Menke} for the definition). Suppose a result of contact surgery on a mixed Legendrian knot $L$ is Liouville fillable. There exists a mixed torus associated to $L$ in $Y$ and it has a \dfn{mixed neighborhood} in $Y$. We show that this neighborhood is still contained in a result of contact surgery on $L$ and because of this neighborhood, we can show that the only possible decomposition of the Liouville filling along the mixed torus is a union of a Liouville filling of $(Y,\xi)$ and a filling of some lens space. By the assumption that $(Y,\xi)$ is not Liouville fillable however, we obtain a contradiction. 

After that, we apply Theorem~\ref{thm:mixed} to prove Theorem~\ref{thm:non-Liouville}. If we perform 0-surgery on $K$ then we obtain a torus bundle over $S^1$. There are \dfn{rotative contact structures} on a torus bundle which are weakly fillable but not strongly fillable. Choose one of them and in this contact structure, we find a mixed Legendrian knot on which negative contact surgery produces a rational homology $3$--sphere. Since negative contact surgery preserves weak fillability, the resulting contact structure is also weakly fillable but also strongly fillable since the second cohomology with $\mathbb{Q}$-coefficient of the manifold is trivial. However, by Theorem~\ref{thm:mixed}, the contact structure does not admit Liouville fillings. Finally, we keep track of the surgery coefficient with respect to the Seifert framing of $K$ using its fractional Dehn twist coefficient. 

\subsection{Contact structures without weak fillings} \label{sec:non-weak}
Theorem~\ref{thm:mixed} tells us that mixed Legendrian knots also preserve weak non-fillability under contact surgery. Using this, we can find various tight contact $3$--manifolds without weak fillings, which were not accessible previously.

\begin{example}\label{ex:not-weak-torus}
  There are tight contact structures on some Seifert fibered spaces without weak fillings, see \cite{EH:nofilling, GLS:noweak, Matkovic:weak, PVHM:weak} for examples. Let $L$ be a mixed Legendrian knot in one of these contact manifolds. Then negative contact $(r)$-surgery on $L$ for any $r < 0$ yields tight contact structures without weak fillings by Theorem~\ref{thm:mixed}.
\end{example}


\begin{example}\label{ex:not-weak-f8}
  Let $K$ be the figure-eight knot. In \cite{CM:figure8}, Conway and the author showed that $S^3_{\pm4}(K)$ admit tight contact structures $\xi_\pm$ containing separating Giroux torsion. By Ghiggini and Honda \cite{GH:torsion}, they are not weakly fillable. \cite[Proposition~3.11 and~3.12]{CM:figure8} implies that $S^3_r(K)$ can be obtained by negative contact surgery on some mixed Legendrian knots in $\xi_\pm$ if $r \in (-4,-15/4) \cup (4,17/4)$. According to Theorem~\ref{thm:mixed}, these hyperbolic rational homology spheres admit tight contact structures without weak fillings. 
\end{example}

\begin{remark}
  In \cite[Theorem~1.4]{CM:figure8}, Conway and the author showed that every tight contact structure on $S^3_r(K)$ for $r \notin (-4,-3) \cup (0,1) \cup (4,5)$ is strongly fillable. Thus $S^3_r(K)$ can admit non-fillable tight contact structure only if $r \in (-4,-3) \cup (0,1) \cup (4,5)$.
\end{remark}


\subsection{Contact structures on \texorpdfstring{$-\Sigma(2,3,6n\pm1)$}{-Sigma(2,3,6n+1)}} \label{sec:Brieskorn}
Ghiggini and Van-Horn-Morris \cite{GVHM:classification} classified tight contact structures on $-\Sigma(2,3,6n-1)$ for $n \geq 2$. They showed there exist ${n(n-1)}/{2}$ tight contact structures
\[
  \left\{\eta^n_{i,j} :\,\, \begin{aligned} &0 \leq i \leq n-2\\ &|j| \leq n - i - 2,\,\, j\equiv n-i \modp2 \end{aligned}  \right\}
\]
and they are all strongly fillable. See below for $n=5$. They also showed that $\eta^n_{0,j}$ are Stein fillable for $|j| \leq n-2$. However, Ghiggini \cite{Ghiggini:nonfillable} showed $\eta^n_{n-2,0}$ are not Liouville fillable for $n\geq3$.

\begin{center}
\begin{tabular}{rccccccc}
  &    &    &    &  $\eta^5_{3,0}$\\\noalign{\smallskip\smallskip}
  &    &    &  $\eta^5_{2,-1}$ &    &  $\eta^5_{2,1}$\\\noalign{\smallskip\smallskip}
  &    &  $\eta^5_{1,-2}$ &    &  $\eta^5_{1,0}$ &    &  $\eta^5_{1,2}$\\\noalign{\smallskip\smallskip}
  &  $\eta^5_{0,-3}$ &    &  $\eta^5_{0,-1}$ &    &  $\eta^5_{0,1}$ &    &  $\eta^5_{0,3}$\\\noalign{\smallskip\smallskip}
\end{tabular}
\end{center}

  Tosun \cite{Tosun:classification} classified tight contact structures on $-\Sigma(2,3,6n+1)$ for $n\geq1$. He showed there exist $n(n+1)/2$ tight contact structures
\[
  \left\{\xi^n_{i,j} :\,\, \begin{aligned} &0 \leq i \leq n-1\\ &|j| \leq n - i - 1,\,\, j\not\equiv n-i \modp2 \end{aligned}  \right\}
\]
and they are all strongly fillable. See below for $n=4$. He also showed that $\xi^n_{0,j}$ are Stein fillable for $|j| \leq n-1$. By applying the same argument of Ghiggini~\cite{Ghiggini:nonfillable}, he showed $\xi^n_{i,0}$ are not Liouville fillable for $1 \leq i \leq n-1$ and $n \geq 2$.

\begin{center}
\begin{tabular}{rccccccc}
  &    &    &    &  $\xi^4_{3,0}$\\\noalign{\smallskip\smallskip}
  &    &    &  $\xi^4_{2,-1}$ &    &  $\xi^4_{2,1}$\\\noalign{\smallskip\smallskip}
  &    &  $\xi^4_{1,-2}$ &    &  $\xi^4_{1,0}$ &    &  $\xi^4_{1,2}$\\\noalign{\smallskip\smallskip}
  &  $\xi^4_{0,-3}$ &    &  $\xi^4_{0,-1}$ &    &  $\xi^4_{0,1}$ &    &  $\xi^4_{0,3}$\\\noalign{\smallskip\smallskip}
\end{tabular}
\end{center}

Ghiggini and Van-Horn-Morris \cite[Conjecture~1.2]{GVHM:classification} conjectured that $\eta^n_{i,j}$ are not Liouville fillable for $i>0$.
We partially answer this conjecture by showing the contact structures on the `inner triangle' are not Liouville fillable. 

\begin{theorem}\label{thm:triangle} For any $n > 3$, we have
  \begin{enumerate}
    \item $\eta^n_{i,j}$ are strongly fillable but not Liouville fillable for $0 < i < n-3$ and $|j| < n - i -2$.
    \item $\xi^n_{i,j}$ are strongly fillable but not Liouville fillable for $0 < i < n-2$ and $|j| < n - i - 1$.
  \end{enumerate}
\end{theorem}


In the forthcoming paper with Etnyre and Tosun \cite{EMT:torus}, we disprove the conjecture by finding Stein fillings of $\eta^4_{1,\pm1}$. Thus the contact structures on the `edges' can be Stein fillable. We suggest a new conjecture as follows.

\begin{conjecture} 
  For any $n > 2$, we have 
  \begin{enumerate}
    \item $\eta^n_{i,\pm(n-i-2)}$ are Stein fillable for $1 \leq i \leq n-3$. 
    \item $\xi^n_{i,\pm(n-i-1)}$ are Stein fillable for $1 \leq i \leq n-2$. 
  \end{enumerate}
\end{conjecture}

\subsection*{Acknowledgements}
The author thanks John Baldwin for the helpful discussion about the fractional Dehn twist coefficient and Austin Christian, John Etnyre, and Michael Menke for the helpful discussions about weak fillability. The author also appreciates Lisa Piccirillo and Agniva Roy for the helpful conversations.

\section{Background}

We assume readers have a basic understanding of $3$--dimensional contact topology, in particular convex surface theory, and Legendrian and transverse knots. See \cite{Honda:classification1} and \cite{Etnyre:knots} for details. We also assume readers are familiar with open book decompositions of closed $3$--manifolds. See \cite{Etnyre:openbook} for details. 

\subsection{Genus one fibered knots and the fractional Dehn twist coefficients}
We say a knot $K$ in a closed oriented $3$--manifold $Y$ \dfn{admits a genus one open book $(\Sigma,\phi)$} if $K$ is a genus one fibered knot, $\Sigma$ is a fiber surface of $K$, and $\phi$ is the monodromy. By a result of Thurston \cite{Thurston:mcg}, $\phi$ is freely isotopic to one of three types: pseudo-Anosov, periodic and reducible. Let $a,b: \Sigma \to \Sigma$ be the monodromies that are positive Dehn twists along homologically essential closed curves on $\Sigma$, respectively, forming a basis of $H_1(\Sigma)$. Also, let $\delta: \Sigma \to \Sigma$ be the monodromy which is a positive Dehn twist along a closed curve parallel to the boundary. Baldwin \cite{Baldwin:fdtc} classified monodromies on $\Sigma$ (up to conjugation).

\begin{theorem}[Baldwin \cite{Baldwin:fdtc}]\label{thm:classification-mcg}
  Suppose $K$ admits a genus one open book $(\Sigma, \phi)$. The monodromy $\phi$ is one of the following types:
  \begin{enumerate}
    \item $\delta^na^{r_1}b^{-1} \cdots a^{r_k}$, where $r_1,\ldots,r_k \geq 0$, and at least one $r_i > 0$.
    \item $\delta^nab^2ab^2a^{r_1}b^{-1} \cdots a^{r_k}b^{-1}$, where $r_1,\ldots,r_k \geq 0$, and at least one $r_i > 0$.
    \item $\delta^na^mb^{-1}$, for $m = -1, -2, -3$. 
    \item $\delta^nab^2ab^2a^mb^{-1}$, for $m = -1, -2, -3$.
    \item $\delta^nb^m$, for $m \in \mathbb{Z}$.
    \item $\delta^nab^2ab^2b^m$, for $m \in \mathbb{Z}$.
  \end{enumerate} 
  Types (1) and (2) are pseudo-Anosov, (3) and (4) are periodic, and (5) and (6) are reducible.
\end{theorem}

We briefly review the fractional Dehn twist coefficient which was introduced by Honda, Kazez and Mati\'c \cite{HKM:right-veering1}. To make our discussion simple, we assume a surface $\Sigma$ has genus one with one boundary component. The \dfn{fractional Dehn twist coefficient} is a map $c: \mathrm{Mod}(\Sigma,\partial\Sigma) \to \mathbb{Q}$ that measures twists of the monodromy along the boundary component.

Let $\alpha$ and $\beta$ be properly embedded arcs on $\Sigma$ having common endpoints. We perturb them to have minimal intersections. We say $\alpha$ is \dfn{to the right} of $\beta$ if the tangent vectors of $\alpha$ and $\beta$ on the endpoints form a positive basis with respect to the orientation of $\Sigma$. We also say $\alpha$ is \dfn{to the left} of $\beta$ if the tangent vectors of $\beta$ and $\alpha$ on the endpoints form a positive basis with respect to the orientation of $\Sigma$. We adopt the convention that $\alpha$ is to the right of $\alpha$ but not to the left of $\alpha$. We say a monodromy $\phi$ is \dfn{right-veering} if $\phi(\alpha)$ is to the right of $\alpha$ for any properly embedded arc $\alpha$ on $\Sigma$. If $\Sigma$ has genus one and one boundary component, the fractional Dehn twist coefficient can (mostly) tell whether the monodromy is right-veering or not.


\begin{proposition}[\cite{Baldwin:fdtc, HKM:right-veering1, HKM:invariants}]\label{prop:fdtc-veering} Suppose $K$ admits a genus one open book $(\Sigma, \phi)$. Then the followings hold.
  \begin{enumerate}
    \item If $\phi$ is pseudo-Anosov, $\phi$ is right-veering if and only if $c(\phi) > 0$.
    \item If $\phi$ is periodic, $\phi$ is right-veering if and only if $c(\phi)\geq0$.
    \item If $\phi$ is reducible, $\phi$ is right-veering if $c(\phi) > 0$ and not right-veering if $c(\phi)<0$.
  \end{enumerate}
\end{proposition}

There are some simple, but useful properties about the boundary twisting monodromy $\delta$ which are used throughout the paper, which can be found in \cite{Baldwin:fdtc,KR:fdtc}

\begin{lemma}\label{lem:fdtc-addition}
  Suppose $K$ admits a genus one open book $(\Sigma, \phi)$. Then the followings hold.
  \begin{enumerate}
    \item $c(\delta^{\pm1} \circ \phi) = c(\phi) \pm 1$,
    \item $\delta = (ab^2ab^2)^2 = (ab)^6$,
    \item $\delta \circ \phi = \phi \circ \delta$.
  \end{enumerate}
\end{lemma}

There are arcs on $\Sigma$ that characterize the monodromies with the small fractional Dehn twist coefficient.

\begin{lemma} \label{lem:arcs}
  Suppose $K$ admits a genus one open book $(\Sigma, \phi)$. Then the followings holds.
  \begin{enumerate}
    \item Suppose $\phi$ is pseudo-Anosov and $\lceil c(\phi) \rceil = 0$. Then there is a properly embedded arc $\alpha$ on $\Sigma$ such that $\phi(\alpha)$ is to the left of $\alpha$, and $\delta \circ \phi (\alpha)$ is to the right of $\alpha$.  
    \item Suppose $\phi$ is not pseudo-Anosov and $\lfloor c(\phi) \rfloor = 0$. Then there is an arc $\alpha$ on $\Sigma$ such that $\phi(\alpha)$ is to the right of $\alpha$, and $\delta^{-1} \circ \phi(\alpha)$ is to the left of $\alpha$.
  \end{enumerate}
\end{lemma}

\begin{proof}
  We first start with the pseudo-Anosov case. According to Proposition~\ref{prop:fdtc-veering}, $\phi$ is not right-veering and there exists a properly embedded arc $\alpha$ on $\Sigma$ such that $\phi(\alpha)$ is to the left of $\alpha$. Since $\delta \circ \phi$ is right-veering by the assumption, $\phi(\alpha)$ is to the right of $\alpha$.
  
  Next, suppose $\phi$ is periodic. According to Proposition~\ref{prop:fdtc-veering}, $\phi$ is right-veering and $\delta^{-1} \circ \phi$ is not right-veering. Thus there exists a properly embedded arc $\alpha$ on $\Sigma$ such that $\phi(\alpha)$ is to the right of $\alpha$, and $\delta^{-1} \circ \phi(\alpha)$ is to the left of $\alpha$.  

  Now suppose $\phi$ is reducible. According to Theorem~\ref{thm:classification-mcg}, $\phi$ is one of $b^m$ and $ab^2ab^2b^m$ for $m \in \mathbb{Z}$. Let $\alpha$ be a properly embedded arc on $\Sigma$ which is fixed by $b$. Then $\phi(\alpha)$ is either $\alpha$ or $ab^2ab^2(\alpha)$, so it is to the right of $\alpha$. However $\delta^{-1} \circ \phi(\alpha)$ is either $\delta^{-1}(\alpha)$ or $(ab^2ab^2)^{-1}(\alpha)$ by Lemma~\ref{lem:fdtc-addition}, so it is to the left of $\alpha$.
\end{proof}

We finish this section by introducing a capping off operation. We can obtain a closed surface $\Sigma_c$ by gluing a disk to $\Sigma$ along the boundary, and obtain a monodromy $\phi_c$ by extending $\phi$ by the identity over the disk. We say $(\Sigma_c,\phi_c)$ is obtained from $(\Sigma,\phi)$ by \dfn{capping off} the boundary of $\Sigma$. It is clear that any two monodromies $\phi$ and $\psi$ that differ by $\delta^k$ become isotopic after capping off the boundary.

\subsection{Decomposition of Liouville and weak fillings along a torus} \label{sec:Menke}
Menke \cite{Menke:decomposition} showed a Liouville domain can be decomposed along a properly embedded solid torus whose boundary is a \dfn{mixed torus}, which we define below. After that, Christian and Menke \cite{CM:decomposition} generalized the result to a mixed surface with any positive genus. These can be considered as a generalization of the result of Eliashberg \cite{Eliashberg:decomposition} that a Weinstein domain can be decomposed along a properly embedded $B^3$ whose boundary is a convex sphere.  

First, we review the Farey graph. Given two rational numbers $a/b$ and $c/d$, we define their \dfn{Farey sum} to be 
\[
  \frac ab \oplus \frac cd = \frac{a+c}{b+d}.
\] 
 We also define their \dfn{Farey multiplication} to be
\[
  \frac ab \bigcdot \frac cd = ad - bc.
\]

Consider the Poincar\'e disk in $\R^2$. Label the point $(0,1)$ by $0=0/1$ and $(0,-1)$ by $\infty=1/0$. Take the half circle with non-negative $x$-coordinate. Pick a point in a half-way between two labeled points and label it with the Farey sum of the two points and connect it to both points by a geodesic. Iterate this process until all the positive rational numbers are a label on some point on the unit disk. Now do the same for the half circle with non-positive $x$-coordinate (for $\infty$, use the fraction $-1/0$). See Figure~\ref{fig:Farey}. 

Notice that for two rational numbers $r$ and $s$, we have $|r \bigcdot s| = 1$ if and only if there is an edge between them in the Farey graph.

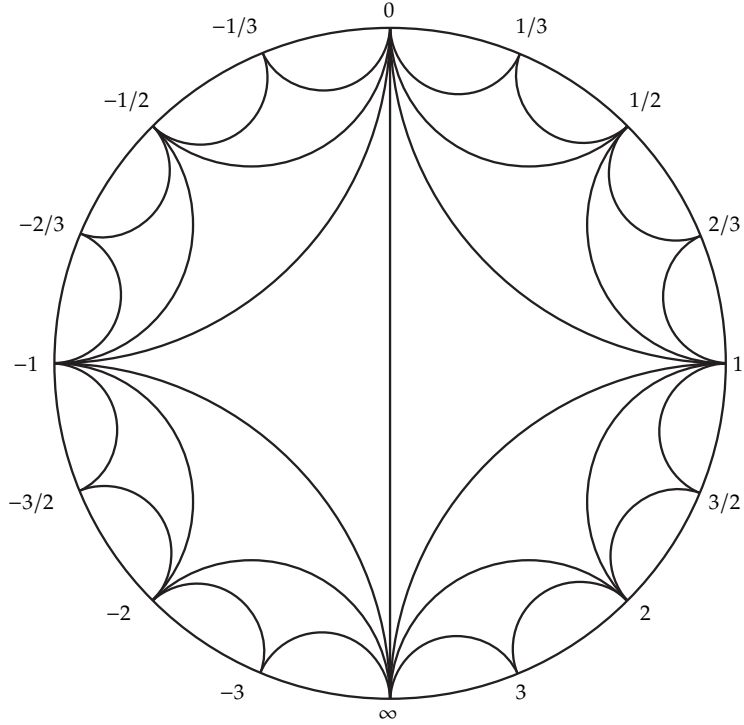
\begin{figure}[htb]{\scriptsize
  \begin{overpic}[tics=20]{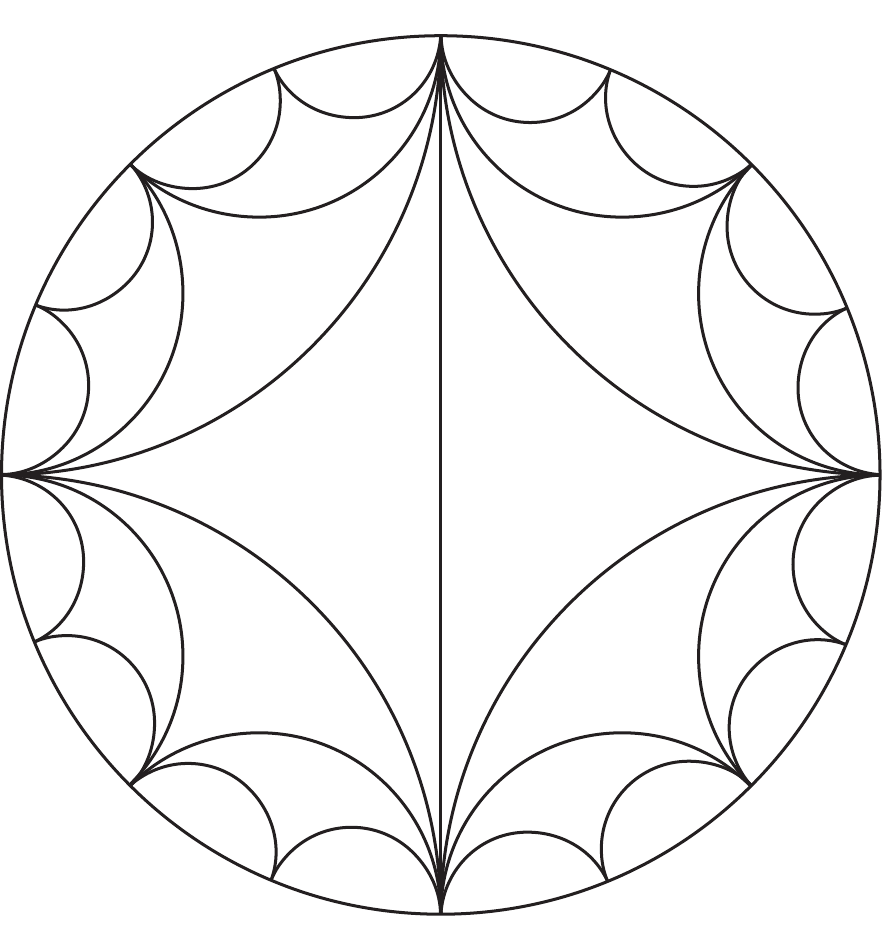}
    \put(123, 0){$\infty$}
    \put(125, 266){$0$}
    \put(-15, 132){$-1$}
    \put(257, 132){$1$}
    \put(20, 38){$-2$}
    \put(222, 38){$2$}
    \put(19, 232){$-1/2$}
    \put(218, 232){$1/2$}
    \put(60, 260){$-1/3$}
    \put(175, 260){$1/3$}
    \put(-13, 185){$-2/3$}
    \put(248, 185){$2/3$}
    \put(-17, 80){$-3/2$}
    \put(248, 80){$3/2$}
    \put(63, 8){$-3$}
    \put(175, 8){$3$}
  \end{overpic}}
  \caption{The Farey graph.}
  \label{fig:Farey}
\end{figure}

A \dfn{mixed torus} in a contact $3$--manifold $(Y,\xi)$ is a convex torus $T$ such that there exists a virtually overtwisted $T \times [-1,1]$ in $(Y,\xi)$ where $T := T\times\{0\}$ and $T\times[-1,0]$ and $T\times[0,1]$ are basic slices with different signs. Let $T_{-1} := T\times\{-1\}$ and $T_1 := T\times\{1\}$. We call this $T^2 \times [-1,1]$ the \dfn{mixed neighborhood} for $T$.  Also let $s_{-1}$, $s_0$ and $s_1$ be the dividing slopes of $T_{-1}$, $T$, and $T_1$, respectively. We require that $s_0$ is clockwise of $s_{-1}$ and anticlockwise of $s_1$ in the Farey graph. 

We say a solid torus with convex boundary is a \dfn{solid torus with lower meridian} if it has two dividing curves with meridional slope $r$ and dividing slope $s$. Also any convex torus in the solid torus parallel to the boundary has dividing slope $t$ such that $t$ is clockwise of $r$ and anticlockwise of $s$ in the Farey graph. We denote it by $S(r,s;l)$.  

We say a solid torus with convex boundary is a \dfn{solid tours with upper meridian} if it has two dividing curves with meridional slope $r$ and dividing slope $s$. Also any convex torus in the solid torus parallel to the boundary has dividing slope $t$ such that $t$ is anticlockwise of $r$ and clockwise of $s$ in the Farey graph. We denote it by $S(r,s;u)$.

Now we are ready to state Menke's result.

\begin{theorem}[Menke \cite{Menke:decomposition}] \label{thm:Menke}
  Let $(Y,\xi)$ be a closed contact $3$--manifold and $(W,\omega)$ be a Liouville (resp.~weak) filling of $(Y,\xi)$. If there exists a mixed torus $T \subset (Y,\xi)$ with a mixed neighborhood $T^2 \times [-1,1]$, then there exists a Liouville (resp.~weak) filling $(W',\omega')$ of $(Y',\xi')$ such that
  \begin{itemize}
    \item $(Y',\xi')$ is obtained by cutting $Y$ along $T$ and gluing two tight solid tori $S(s,s_0;l)$ and $S(s,s_0;u)$ to $T\times[0,1]$ and $T\times[-1,0]$ along $T$, respectively, for some $s \in \mathbb{Q}\cup\{\infty\}$ such that $s$ is anticlockwise of $s_{-1}$ and clockwise of $s_1$ in the Farey graph and $|s \bigcdot s_0| = 1$. 
    \item $(W',\omega')$ is obtained by cutting $W$ along a properly embedded solid torus $S$ with $\partial S = T$.
  \end{itemize}
\end{theorem}


\begin{remark}
  In \cite{Menke:decomposition, CM:decomposition}, the theorems were proved for Liouville and strong fillings. However, the proofs actually work for Liouville and weak fillings.
\end{remark}

\subsection{Contact and transverse surgery} \label{sec:surgery}
In this section, we review two surgery operations on contact manifolds: \dfn{contact surgery} and \dfn{transverse surgery}. 

Let $(Y, \xi)$ be a contact $3$--manifold and $L$ be a Legendrian knot in $(Y,\xi)$. A \dfn{contact surgery} on $L$ is a surgery operation to produce another contact manifold. Here we briefly review the construction. For more details, see \cite{DG:surgery}. 

First, recall that $S(r,s;l)$ is a convex solid torus with lower meridian having the meridional slope $r$ and the dividing slope $s$, see Section~\ref{sec:Menke}. Here, we use the contact framing, which is induced from the contact planes along $L$. To perform \dfn{contact $(r)$--surgery}, take a standard neighborhood $N$ of $L$. Notice that $N$ is contactomorphic to $S(\infty,0;l)$ (there exists a unique tight contact structure on $S(\infty,0;l)$). We remove $N$ from $(Y, \xi)$ and glue $S(r,0;l)$ to the boundary of $Y \setminus N$ and extend the contact structure to the entire $S(r,0;l)$ so that the restriction of the contact structure to $S(r,0;l)$ is tight. Notice that there exists a unique tight contact structure on $S(r,0;l)$ if and only if $r = 1/n$ for $n \in \mathbb{Z}$. Thus in general, contact $(r)$--surgery can produce different contact manifolds depending on the choice of a tight contact structure on $S(r,0;l)$. Contact $(-1)$--surgery is also called \dfn{Legendrian surgery}.

\begin{remark}
  We use parenthesis for contact surgery coefficients to indicate that we use the contact framing.
\end{remark}

Let $(Y, \xi)$ be a contact $3$--manifold and $K$ be a transverse knot in $(Y,\xi)$ with a fixed a framing. A {\it transverse surgery} on $K$ is a surgery operation to produce another contact manifold. Here, we briefly review two types of transverse surgeries: \dfn{admissible} and \dfn{inadmissible} transverse surgeries. For more details, see \cite{BE:transverse,Conway:transverse}.

Choose a neighborhood $N$ of $K$ with convex boundary. Suppose $N$ is a union of standard neighborhood of a Legendrian push-off of $K$ and negative basic slices. Let $s$ be the dividing slope of $\partial N$. Then for any $a < s$, there is another neighborhood of $K$ inside $N$, contactomorphic to $S_a = \{r \leq \theta_a\} \subset (S^1 \times \R^2, \xi_{\rm{rot}}=\ker(\cos r\, d\phi + r\sin r\,d\theta))$ where $0 < \theta_a < \pi$, $\tan \theta_a = -1/a$, and $K$ is identified with the $\phi$-axis. The slope of the leaves of the characteristic foliation on the torus $\{r = \theta_a\}$ is $a$.

To perform \dfn{admissible transverse $p/q$--surgery}, pick a neighborhood $S_{p/q}$ inside $N$ and remove the interior of $S_{p/q}$. Then perform a contact cut on the boundary, {\it i.e.\ }quotient the boundary by the $S^1$-action of the leaves of the characteristic foliation (see \cite{Lerman:contact-cut} for details). The resulting contact manifold after the quotient gives the result of admissible transverse $p/q$-surgery on $K$.  

To perform \dfn{inadmissible transverse $p/q$--surgery}, pick a neighborhood $S_{a}$ inside $N$ for any $a<s$ and remove the interior of $S_{a}$. Then glue the $T^2 \times I$ layer $\{\theta_{p/q} \leq r \leq \theta_a\}$ to $\partial(Y \setminus S_{a})$ along $\{r=\theta_a\}$ where $0< \theta_a - \theta_{p/q} < \pi$. Then perform a contact cut on the new boundary and obtain the result of inadmissible transverse $p/q$-surgery on $K$.

Unlike a contact surgery on a Legendrian knot, the resulting contact structure of admissible transverse surgery depends on the choice of a neighborhood of $K$. Thus in this paper, we always specify a neighborhood $N$ of a transverse knot $K$ and we say \dfn{transverse surgery on $K$ using $N$} or \dfn{transverse surgery using $N$} in short.

We finish this section by introducing a useful proposition which is used throughout the paper. 

\begin{proposition}[Conway \cite{Conway:transverse}] \label{prop:keep-track-of-slope}
  Suppose $K$ admits a genus one open book $(\Sigma,\phi)$. With respect to the page framing, $\pm1$-surgery on $K$ produces a new open book $(\Sigma, \delta^{\mp1} \circ \phi)$. 
  
  Let $\xi$ be a contact structure on $Y$ and $N$ be a neighborhood of $K$ with convex boundary of slope $p/q$. Suppose we can perform admissible transverse $-1$--surgery using $N$. After the surgery, the dividing slope of $N$ with respect to the page framing changes as follows: 
  \[
    \begin{pmatrix}
      1 & 1\\
      0 & 1
    \end{pmatrix}
    \begin{pmatrix}
      q\\p
    \end{pmatrix}
    =
    \begin{pmatrix}
      q+p\\p
    \end{pmatrix}.
  \]
  If we perform inadmissible transverse $1$--surgery using $N$, the dividing slope of $N$ with respect to the page framing changes as follows:
  \[
    \begin{pmatrix}
      1 & -1\\
      0 & 1
    \end{pmatrix}
    \begin{pmatrix}
      q\\p
    \end{pmatrix}
    =
    \begin{pmatrix}
      q-p\\p
    \end{pmatrix}.
  \]
\end{proposition}

\section{Rotative contact structures on torus bundles} \label{sec:torus-bundles} 

In this section, we construct a neighborhood of genus one fibered knot with certain properties. To do so, we first review rotative contact structures on torus bundles over $S^1$.

Let $M_\phi$ be a $T^2$-bundle over $S^1$ with the monodromy $\phi$. We may identify $T^2 = \mathbb{R}^2 / \mathbb{Z}^2$ and $\phi = A \in \mathrm{SL}_2(\mathbb{Z})$. Then for each $n \in \mathbb{N}\cup\{0\}$, a \dfn{rotative contact torus bundle $(M_\phi, \xi_\phi^n)$} is
\begin{gather*}
  M_{\phi} = T^2 \times \mathbb{R} / (\mathrm{x},t) \sim (A\mathrm{x},t-1),\\ 
  \xi_{\phi} = \ker (\sin f(t)\,dx + \cos f(t)\,dy),
\end{gather*}
where $\mathrm{x} = (x,y) \in T^2$ and $f:\mathbb{R} \to \mathbb{R}$ satisfying
\begin{enumerate}
  \item $f'(t) > 0$ for any $t \in \mathbb{R}$,
  \item $\xi_\phi^n$ is invariants under the action $(\mathrm{x},t) \mapsto (A\mathrm{x},t-1)$,
  \item $2n\pi < \displaystyle\sup_{t\in\mathbb{R}}(f(t+1) - f(t)) \leq 2(n+1)\pi$.
\end{enumerate}

Giroux \cite{Giroux:classification1,Giroux:classification2,Giroux:classification3} and Honda \cite{Honda:classification2} classified rotative contact structures on every torus bundle.

\begin{theorem}[Giroux \cite{Giroux:classification1,Giroux:classification2,Giroux:classification3}, Honda \cite{Honda:classification2}]
  Let $M_\phi$ be a torus bundle over $S^1$. For any $n \geq 0$, there exists unique rotative contact structure $\xi_\phi^n$ up to contactomorphism.  
\end{theorem}

Eliashberg \cite{Eliashberg:nostrong} determined the fillability of rotative contact structures on $T^3$. After that, Gay \cite{Gay:fillability} completely determined the fillability of rotative contact structures on every torus bundle when $n \geq 1$.

\begin{theorem}[Gay \cite{Gay:fillability}, see also \cite{Eliashberg:nostrong}] \label{thm:torus-fillability}
  $\xi_\phi^n$ is weakly fillable but not strongly fillable for $n \geq 1$.
\end{theorem}

\begin{remark}
  If $n\geq1$, a rotative contact structure $\xi_\phi^n$ contains Giroux torsion, which is planar $1$-torsion.
\end{remark}

Now for any genus one fibered knot $K$ in a closed $3$--manifold $Y$, we construct a contact structure $\xi_n$ on $Y$ for $n \geq 0$, and a neighborhood $N$ of $K$ contained in $(Y,\xi_n)$ with several properties. We can consider the following proposition as a generalization of \cite[Lemma~3.5]{Ghiggini:invariants}.

\begin{proposition} \label{prop:measure-slope}
  Suppose $K$ is a genus one fibered knot in a closed oriented $3$--manifold $Y$ and $\phi$ is the monodromy of $K$. Then there exists a contact structure $\xi_n$ on $Y$ for $n \geq 0$ such that 
  \begin{enumerate}
    \item if $\phi$ is pseudo-Anosov, then there is a neighborhood $N$ of $K$ such that $\partial N$ is a convex torus with two dividing curves of slope $1/\lceil c(\phi) + n \rceil$. 
    \item if $\phi$ is not pseudo-Anosov, then there is a neighborhood $N$ such that $\partial N$ is a convex torus with two dividing curves of slope $1/\lfloor c(\phi) + n + 1 \rfloor$. 
  \end{enumerate}
  Also, if we perform an admissible transverse $0$--surgery using $N \subset (Y,\xi_n)$, then we obtain a rotative torus bundle $(M_{\phi_c}, \xi_{\phi_c}^n)$, where $\phi_c$ is obtained from $\phi$ by capping off the boundary.
\end{proposition}

\begin{proof}
  We start with the pseudo-Anosov case first. Suppose $\lceil c(\phi) \rceil = 0$ and $n=0$. Let $\Sigma$ be a fiber surface of $K$. We cap off the boundary of $\Sigma$ and obtain a torus $T$ and the monodromy $\phi_c$. By Lemma~\ref{lem:arcs}, there is a properly embedded arc $\alpha$ on $\Sigma$ such that $\phi(\alpha)$ is to the left of $\alpha$ and $\delta\circ\phi(\alpha)$ is to the right of $\alpha$. Let $\alpha_c$ be a closed curve on $T$ obtained by capping off the boundary of $\Sigma$. Now we consider a rotative torus bundle $(M_{\phi_c},\xi_{\phi_c}^0)$. We can choose the identification $T = \mathbb{R}^2/\mathbb{Z}^2$, $\phi_c=A\in\mathrm{SL}_2(\mathbb{Z})$, and $f:\mathbb{R}\to\mathbb{R}$ so that the characteristic foliation of $T\times\{1\}$ has slope $0$ and the leaves of this characteristic foliation is isotopic to $\alpha_c$.

  Consider $(0,0) \in T = \mathbb{R}^2/\mathbb{Z}^2$ and let $D$ be a small disk centered at $(0,0)$. Perturb $A$ to be the identity on $D$. Here, we choose a perturbation of $A$ so that we match the characteristic foliation on $T\times\{0\}$ and $A(T\times\{1\})$ by perturbing the characteristic foliation on $T\times\{0\}$ in a clockwise way near $D$. See the bottom right drawing of Figure~\ref{fig:foliations}. Consider a solid torus $N = D \times I / (\mathrm{x},1) \sim (A\mathrm{x},0)$, which is a standard neighborhood of a Legendrian knot $L := (0,0) \times I / (0,0,1) \sim (0,0,0)$. Notice that the dividing arc on $\partial D \times I$ is $t \mapsto (e^{i\theta t}, t)$, where $0 < \theta = f(1) - f(0) \leq 2\pi$. Since we rotate the characteristic foliation on $D\times\{0\}$ in a clockwise way, the dividing slope of $\partial N$ is the same as the product framing. See Figure~\ref{fig:slope}. If $n > 0$, then the contact structure is $\xi_{\phi_c}^n$ and the dividing arc on $\partial D \times I$ is $t \mapsto (e^{i(2n\pi + \theta) t}, t)$ so the dividing slope of $N$ is $-n$ with respect to the product framing. 

  \begin{figure}[htbp]
    \vspace{0.2cm}
    \begin{overpic}[tics=20]{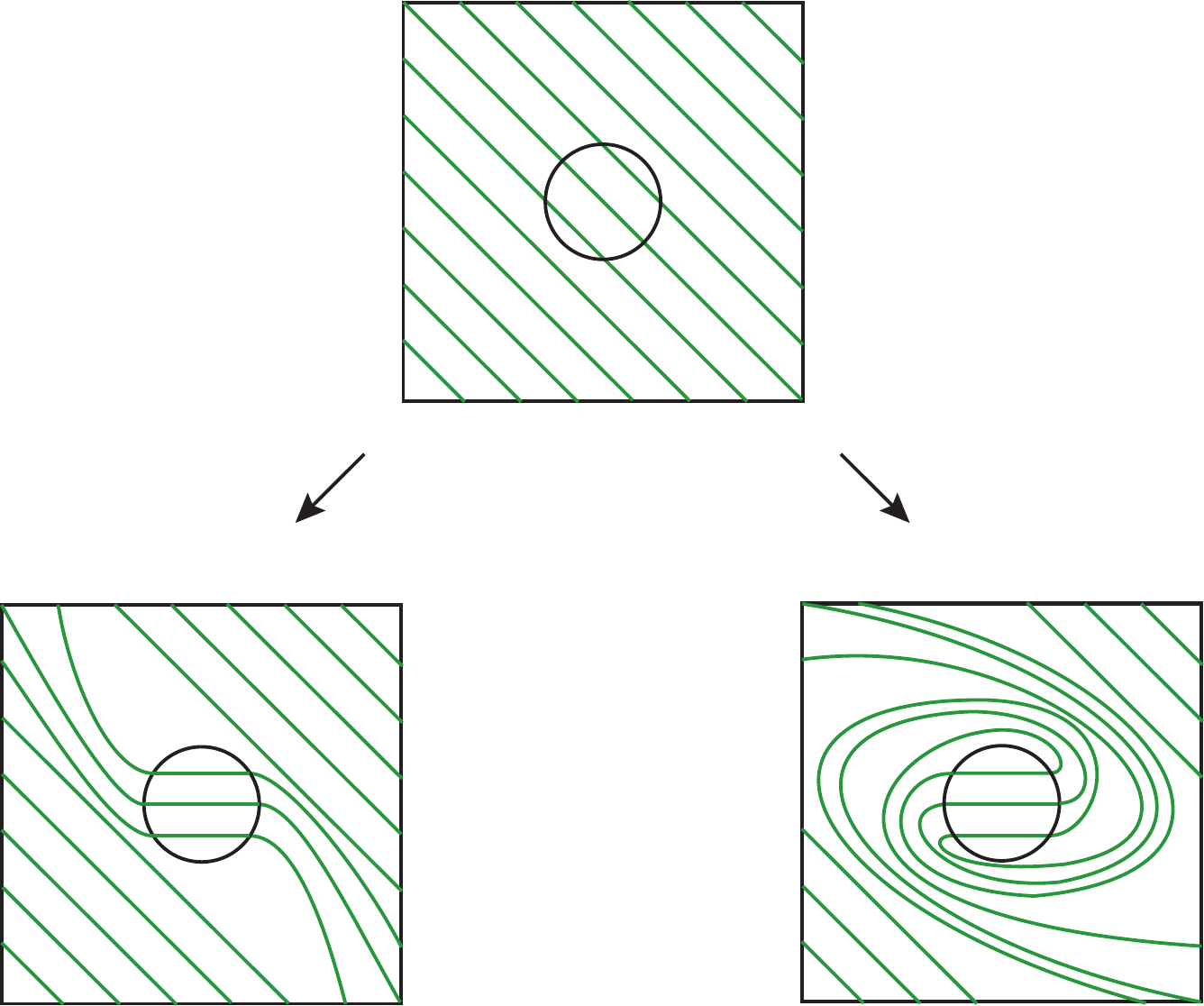}
    \end{overpic}
    \vspace{0.2cm}
    \caption{Two different ways to perturb the characteristic foliation.}
    \label{fig:foliations}
  \end{figure} 

  Choose a leaf of the characteristic foliation on $T\times\{1\}$ which passes $(0,0)$ and let $\alpha'$ be the part of the leaf restricted to $T \setminus D$. Now perform inadmissible transverse $0$-surgery using $N$ with respect to the product framing and we obtain an open book decomposition. The page of this open book decomposition is $T \setminus D$, which can be identified with $\Sigma$ so that $\alpha'$ is isotopic to $\alpha$. The monodromy of the resulting manifold is $A|_{T\setminus D}=\delta^m \circ \phi$ for some $m \in \mathbb{Z}$. Since we rotate the characteristic foliation near $D$ in a clockwise way, $\delta^m \circ \phi(\alpha)$ is to the left of $\alpha$ and $\delta^{m+1} \circ \phi(\alpha)$ is to the right of $\alpha$. We already know that $\phi(\alpha)$ is to the left of $\alpha$ and $\delta \circ \phi(\alpha)$ is to the right of $\alpha$, so $m = 0$. Thus the resulting manifold should be $Y$ and the binding of the open book is $K$. Let $\xi_n$ be the resulting contact structure. Since the product framing becomes a meridional slope of $K$, we can keep track of the dividing slope of $N$ as follows:  
  \[
    \begin{pmatrix}
      0 & -1\\
      1 & 0
    \end{pmatrix}
    \begin{pmatrix}
      1\\
      -n
    \end{pmatrix}  
    = 
    \begin{pmatrix}
      n\\
      1
    \end{pmatrix},
  \] 
  so the dividing slope of $N \subset (Y,\xi_n)$ is $1/n$. If we perform admissible transverse $0$-surgery using $N$ with respect to the page framing, we recover the rotative torus bundle $(M_{\phi_c},\xi_{\phi_c}^n)$. Now assume $\lceil c(\phi) \rceil \neq 0$. Then there exists $\psi$ such that $\lceil c(\psi) \rceil = 0$ and $\phi = \delta^k \circ \psi$ for $k = \lceil c(\phi) \rceil$. By the same argument as above, there are contact structures $\xi'_n$ for $n \geq 0$ on $(\Sigma,\psi)$ such that there is a neighborhood $N$ of the binding in $\xi'_n$ with the dividing slope $1/n$. If $k > 0$, perform admissible transverse $-1/k$--surgery using $N$ with respect to the page framing. Then by Proposition~\ref{prop:keep-track-of-slope}, the monodromy becomes $\delta^k \circ \psi = \phi$ so the resulting manifold is $Y$. Denote the resulting contact structure by $\xi_n$. By Proposition~\ref{prop:keep-track-of-slope}, the dividing slope of $N$ becomes
  \[
    \begin{pmatrix}
      1 & 1\\
      0 & 1
    \end{pmatrix}^k
    \begin{pmatrix}
      n\\
      1
    \end{pmatrix}  
    = 
    \begin{pmatrix}
      k+n\\
      1
    \end{pmatrix}.
  \]
  Thus the dividing slope of $N \subset (Y,\xi_n)$ is $1/(k+n) = 1/\lceil c(\phi) + n \rceil$. If $k < 0$, then perform inadmissible transverse $-1/k$--surgery. Then by Proposition~\ref{prop:keep-track-of-slope}, the monodromy becomes $\delta^k \circ \psi = \phi$ so the resulting manifold is $Y$. Denote the resulting contact structure by $\xi_n$. By Proposition~\ref{prop:keep-track-of-slope}, the dividing slope of $N$ becomes 
  \[
    \begin{pmatrix}
      1 & -1\\
      0 & 1
    \end{pmatrix}^{-k}
    \begin{pmatrix}
      n\\
      1
    \end{pmatrix}  
    = 
    \begin{pmatrix}
      k+n\\
      1
    \end{pmatrix}.
  \]
  Thus the dividing slope of $N \subset (Y,\xi_n)$ is $1/(k+n) = 1/\lceil c(\phi) + n \rceil$. If we perform admissible transverse $0$--surgery using $N \subset (Y,\xi_n)$, we recover the rotative torus bundle $(M_{\phi_c},\xi_{\phi_c}^n)$ (notice that $\phi_c=\psi_c$). 

  \begin{figure}[htbp]
    \vspace{0.2cm}
    \begin{overpic}[tics=20]{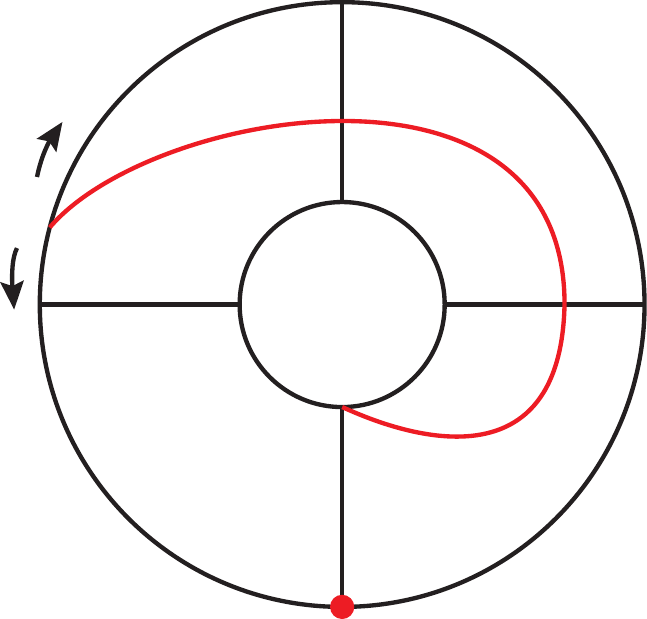}
    \end{overpic}
    \vspace{0.2cm}
    \caption{$\partial D \times I \subset (M_{\phi_c},\xi_{\phi_c}^0)$. If we rotate the red arc in a clockwise way so that the endpoint matches the red point, then we obtain dividing curves of slope $0$ with respect to the product framing. If we rotate the red arc in an anticlockwise way, then we obtain dividing curves of slope $-1$ with respect to the product framing.}
    \label{fig:slope}
  \end{figure}

  Now suppose $\phi$ is not pseudo-Anosov and $\lfloor c(\phi) \rfloor = 0$ Then by Lemma~\ref{lem:arcs}, there is an arc $\alpha$ on $\Sigma$ such that $\phi(\alpha)$ is to the right of $\alpha$ and $\delta^{-1} \circ \phi(\alpha)$ is to the left of $\alpha$. Consider a rotative torus bundle $(M_{\phi_c}, \xi^n_{\phi_c})$. We construct a neighborhood $N$ of $L = (0,0) \times I / (0,0,1) \sim (0,0,0)$ as above but this time, we perturb $A$ in a different way so that we rotate the characteristic foliation near $D \times \{0\}$ in an anticlockwise way to match the characteristic foliations of $T\times\{0\}$ and $A(T\times\{1\})$. See the bottom left drawing of Figure~\ref{fig:foliations}. Then the dividing arc on the $\partial D \times I$ is $t \mapsto (e^{i(2n\pi+\theta)t}, t)$, where $2n\pi < 2n\pi + \theta = f(1) - f(0) \leq 2(n+1)\pi$. Since we rotate $D$ in an anticlockwise way, the dividing slope of $N$ is $-n-1$ with respect to the product framing. See Figure~\ref{fig:slope}. Now perform inadmissible transverse $0$--surgery using $N$ with respect to the product framing and we obtain an open book decomposition. Again, the page is $T \setminus D$ and the monodromy is $\delta^m \circ \phi$ for $m \in \mathbb{Z}$. Take a leaf of the characteristic foliation on $T\times\{1\}$ passing $(0,0)$ and let $\alpha'$ be its restriction to $T \setminus D$. We can identify $T \setminus D$ with $\Sigma$ so that $\alpha$ is isotopic to $\alpha'$. Since we rotate the characteristic foliation near $D\times\{0\}$ in an anticlockwise way, $\delta^m \circ \phi(\alpha)$ is to the right of $\alpha$, and $\delta^{m-1} \circ \phi(\alpha)$ is to the left of $\alpha$. We already know that $\phi(\alpha)$ is to the right of $\alpha$ and $\delta^{-1} \circ \phi(\alpha)$ is to the left of $\alpha$, so $m=0$. Thus the resulting manifold is $Y$ and $K$ is the binding of the open book constructed from the surgery. Let $\xi_n$ be the resulting contact structure. Since the product framing becomes meridional slope of $K$, the dividing slope of $N \subset Y$ becomes $1/(n+1)$. If we perform admissible transverse $0$--surgery using $N \subset (Y,\xi_n)$, then we recover the rotative torus bundle $(M_{\phi_c},\xi^n_{\phi_c})$. Now assume $\lfloor c(\phi) \rfloor \neq 0$. Then there is $\psi$ such that $\lfloor c(\psi) \rfloor = 0$ and $\phi = \delta^k \circ \psi$ for $k = \lfloor c(\phi) \rfloor$. By the same argument as above, there are contact structures $\xi'_n$ for $n\geq0$ on $(\Sigma,\psi)$ such that there is a neighborhood $N$ of the binding and the dividing slope is $1/(n+1)$. Perform (in)admissible transverse $-1/k$-surgery using $N$. Then by Proposition~\ref{prop:keep-track-of-slope}, the monodromy becomes $\delta^k \circ \psi = \phi$ so the resulting manifold is $Y$. Denote the resulting contact structure by $\xi_n$. The dividing slope of $N$ becomes 
  \[
    \begin{pmatrix}
      1 & \pm1\\
      0 & 1
    \end{pmatrix}^{\pm k}
    \begin{pmatrix}
      n+1\\
      1
    \end{pmatrix}  
    = 
    \begin{pmatrix}
      k+n+1\\
      1
    \end{pmatrix}.
  \]
  Thus the dividing slope of $N \subset (Y,\xi_n)$ is $1/(k+n+1) = 1/\lfloor c(\phi)+ n + 1 \rfloor$. If we perform admissible transverse $0$--surgery using $N \subset (Y,\xi_n)$, we recover the rotative torus bundle $(M_{\phi_c},\xi_{\phi_c}^n)$.
\end{proof}

\section{Obstructing Liouville and weak fillability}

In this section, we prove the theorems in Section~\ref{sec:non-liouville}. We first prove Proposition~\ref{prop:mixed}, which is a simple version of Theorem~\ref{thm:mixed}. Then we prove Theorem~\ref{thm:mixed}, Corollary~\ref{cor:planar-torsion}. After that, using the results from Section~\ref{sec:torus-bundles}, we prove Theorem~\ref{thm:non-Liouville} and Theorem~\ref{thm:weak-non-Liouville}.

We say a Legendrian knot $L$ is a \dfn{mixed Legendrian knot} if $L = S_+S_-(L')$ for some Legendrian knot $L'$. In other words, $L$ is a doubly stabilized Legendrian knot with different signs. Let $N(L)$ be a standard neighborhood of $L$. Then $T = \partial N(S_-(L'))$ is the mixed torus for $L$ and its mixed neighborhood is $N(L') \setminus N(L)$. Notice that $T_{-1} = \partial N(L)$ and $T_1 = \partial N(L')$.

The main idea of obstructing Liouville and weak fillability is as follows. Suppose $(W,\omega)$ is a Liouville (resp. weak) filling of $(Y,\xi)$ and $T$ be a mixed torus in $(Y,\xi)$. If we cut $W$ along a properly embedded solid torus $S$ with $\partial S = T$, then we obtain a new Liouville filling $(W',\omega')$ of $(Y',\xi')$. If any component of $(Y',\xi')$ turns out to be not Liouville (resp. weakly) fillable, then the original contact manifold $(Y,\xi)$ cannot be Liouville (resp. weakly) fillable. In general, there exist several possible decompositions and it is hard to determine which gives the genuine decomposition. However, if a contact manifold is obtained from a contact surgery on a mixed Legendrian knot, then there exists a unique possible decomposition.

\begin{proposition} \label{prop:mixed}
  Suppose $L$ is a mixed Legendrian knot in a closed contact $3$--manifold $(Y,\xi)$. Let $(Y_L,\xi_L)$ be the result of Legendrian surgery on $L$. If $\xi$ is not Liouville (resp. weakly) fillable, then $\xi_L$ is also not Liouville (resp. weakly) fillable. 
\end{proposition}

\begin{proof}
  Since $L$ is a mixed Legendrian knot, $L = S_+S_-(L')$ for some Legendrian knot $L'$. We use the contact framing of $S_-(L')$ as our prescribed framing. With respect to this framing, the dividing slope of the mixed torus $T$ for $L$ is $0$, and the slopes of $T_{-1}$ and $T_{1}$ are $-1$ and $1$, respectively. Since this framing is one greater than the contact framing of $L$, the surgery coefficient becomes $-2$ (recall that Legendrian surgery is contact $(-1)$--surgery). Thus $Y_L$ contains $T$, $T_{-1}$ and $T_{1}$. Suppose $(Y_L, \xi_L)$ is Liouville fillable. Then by Theorem~\ref{thm:Menke}, we can cut $Y_L$ along $T$ and glue two solid tori $S(s,0;l)$ and $S(s,0;u)$ to each boundary component for some $s\in\mathbb{Q}\cup\{\infty\}$ and obtain a Liouville fillable contact manifold $(Y',\xi')$. However, $s = \infty$ is the only (extended) rational number satisfying 
  \[
    |s| > 1 \text{ and } \left|s \bigcdot \frac01\right| = 1
  \]  
  ($|s|>1$ implies that $s$ is anticlockwise of $-1$ and clockwise of $1$, and vice versa). Since the surgery coefficient is $-2$, we have the following decomposition: 
  \[
    Y_L \setminus T = (Y_L \setminus S(-2,0;l)) \sqcup S(-2,0;l) 
  \] 
  and $Y \setminus S(\infty,0;l)$ contains $T_1$. Since $Y_L \setminus S(-2,0;l) = Y \setminus S(\infty,0;l)$, we have 
  \begin{gather*}
    Y = (Y_L \setminus S(-2,0;l)) \cup S(\infty,0;l)\\
    S^3 = S(-2,0;l) \cup S(\infty,0;u) 
  \end{gather*}
  Since $S(\infty,0;l)$ has a unique tight contact structure, $(Y',\xi')$ is a union of $(Y,\xi)$ and $(S^3,\xi_{std})$. Since $(Y,\xi)$ is not Liouville (resp.~weakly) fillable, it contradicts that $(Y',\xi')$ is Liouville (resp.~weakly) fillable.
\end{proof}

\begin{remark}
  Proposition~\ref{prop:mixed} implies that a Weinstein cobordism from a strongly fillable contact manifolds obtained by attaching a Weinstein $2$-handle along a mixed Legendrian knot does not improve the level of fillability. 
\end{remark}

Now we generalize the proposition to any contact surgery.

\begin{proof}[Proof of Theorem~\ref{thm:mixed}]
  We first deal with the case $r \geq 0$.  Since $L$ is a mixed Legendrian knot, $L = S_+S_-(L')$ for some Legendrian knot $L'$. We use the contact framing of $S_-(L')$ as our prescribed framing. With respect to this framing, the dividing slope of the mixed torus $T$ for $L$ is $0$ and the slopes of $T_{-1}$ and $T_1$ are $-1$ and $1$, respectively. Since this framing is one greater than the contact framing of $L$, the surgery coefficient becomes $r-1$. To perform contact $(r)$--surgery, we remove a standard neighborhood of $L$ and glue a basic slice $B$ with slopes $\infty$ and $-1$ to the boundary. After that, we attach some tight $S(r-1,\infty;l)$ to the new boundary. Since $|\infty \bigcdot 1| = 1$, the union $B \cup T^2 \times [-1,1]$ becomes also a basic slice with slopes $\infty$ and $1$. Since $T^2 \times [-1,1]$ is virtually overtwisted however, it contradicts that any basic slice is universally tight. Thus the union becomes overtwisted and contact $(r)$--surgery produces an overtwisted contact structure.
  
  Now suppose $r < 0$. Since $r-1 < -1$, the tori $T$, $T_{-1}$, and $T_1$ are still contained in $Y_{(r)}$. Suppose $\xi_{(r)}$ is Liouville fillable. Then by Theorem~\ref{thm:Menke}, we can cut $Y_{(r)}$ along $T$ and glue two solid tori $S(s,0;l)$ and $S(s,0;u)$ to each boundary component for some $s\in\mathbb{Q}\cup\{\infty\}$ and obtain a Liouville fillable contact manifold $(Y',\xi')$. However, $s = \infty$ is the only (extended) rational number satisfying 
  \[
    |s| > 1 \text{ and } \left|s \bigcdot \frac01\right| = 1
  \]  
  ($|s|>1$ implies that $s$ is anticlockwise of $-1$ and clockwise of $1$, and vice versa). Since the surgery coefficient is $r-1$, we have the following decomposition: 
  \[
    Y_{(r)}(L) \setminus T = (Y \setminus S(r-1,0;l)) \sqcup S(r-1,0;l)
  \]  
  and $Y \setminus S(\infty,0;l)$ contains $T_1$. Since $Y_{(r)}(L) \setminus S(r-1,0;l) = Y \setminus S(\infty,0;l)$, we have 
  \begin{gather*}
    Y = (Y_{(r)}(L) \setminus S(r-1,0;l)) \cup S(\infty,0;l)\\
    L(p,q) = S(r-1,0;l) \cup S(\infty,0;u) \;\;\text{where}\;\; \frac pq = \frac{1}{r-1}. 
  \end{gather*} 
  Since $S(\infty,0;l)$ has a unique tight contact structure, $(Y',\xi')$ is a union of $(Y,\xi)$ and some lens space. Since $(Y,\xi)$ is not Liouville (resp.~weakly) fillable, it contradicts that $(Y',\xi')$ is Liouville (resp.~weakly) fillable. 
\end{proof}

\begin{proof}[Proof of Corollary~\ref{cor:planar-torsion}]
  Wendl \cite{Wendl:hierarchy} showed that if $(Y,\xi)$ contains a planar $k$-torsion for $k \geq 0$, then it is not strongly fillable. Thus by Theorem~\ref{thm:mixed}, $(Y_{(r)},\xi_{(r)})$ is not Liouville fillable for any $r$. 
\end{proof}

Now we are ready to prove Theorem~\ref{thm:non-Liouville}. We use Proposition~\ref{prop:measure-slope} to find correct surgery coefficients and construct strongly fillable contact structures. Then we use Theorem~\ref{thm:mixed} to obstruct Liouville fillability. 

\begin{proof}[Proof of Theorem~\ref{thm:non-Liouville}]
  We start with the pseudo-Anosov case. By Proposition~\ref{prop:measure-slope}, there exists a contact structure $\xi_1$ on $Y$ and a neighborhood $N$ of $K$ with dividing slope $1 / \lceil c(\phi) + 1 \rceil$. We can perform admissible transverse $0$--surgery using $N$ and obtain a rotative torus bundle $(M_{\phi_c},\xi^1_{\phi_c})$, which is weakly fillable but not strongly fillable by Theorem~\ref{thm:torus-fillability}. After the surgery, we can calculate the dividing slope of $N$ with respect to the product framing of $M_{\phi_c}$ as follows: 
  \[
    \begin{pmatrix}
      0 & 1\\
      -1 & 0
    \end{pmatrix}
    \begin{pmatrix}
      \lceil c(\phi) + 1 \rceil \\ 1
    \end{pmatrix}
    =
    \begin{pmatrix}
      1 \\ -\lceil c(\phi) + 1 \rceil 
    \end{pmatrix}.
  \]
  Let $L'$ be the core Legendrian knot of $N \subset (M_{\phi_c},\xi_{\phi_c}^1)$, and $L = S_+S_-(L')$. By Theorem~\ref{thm:mixed}, any contact $(r)$--surgery on $L$ for $r < 0$ produces a contact structures without Liouville fillings. we will convert the contact surgery coefficient $(r)$ into a surgery coefficient for $K$. First, notice that the contact framing of $L$ is $-\lceil c(\phi) \rceil - 3$ less than the product framing. Thus the contact surgery coefficient $(r)$ becomes $r -\lceil c(\phi) \rceil - 3$ with respect to the product framing. Now we calculate this surgery coefficient with respect to the Seifert framing of $K$ as follows:
  \[
    \begin{pmatrix}
      0 & -1\\
      1 & 0
    \end{pmatrix}
    \begin{pmatrix}
      1\\
      r -\lceil c(\phi) \rceil - 3
    \end{pmatrix}  
    = 
    \begin{pmatrix}
      \lceil c(\phi) \rceil + 3 - r \\
      1
    \end{pmatrix}.
  \]
  Since $r \in (-\infty, 0)$, the surgery coefficient is in one of the following intervals: 
  \[
    \frac{1}{\lceil c(\phi) \rceil + 3 - r} \in
    \begin{cases}  
      (0,\frac{1}{\lceil c(\phi) \rceil + 3}) &\text{if } \lceil c(\phi) \rceil > -3,\\
      (0,\infty) &\text{if } \lceil c(\phi) \rceil = -3,\\
      (0,\infty] \cup (-\infty,\frac{1}{\lceil c(\phi) \rceil + 3}) &\text{if } \lceil c(\phi) \rceil < -3, 
    \end{cases}
  \]
  which is just $\mathcal{R}(1/n_K)$.

  Since $Y$ is a rational homology $3$--sphere and $K$ is a null-homologous knot, $Y_r(K)$ is also a rational homology $3$--sphere if $r\neq0$. Since $0 \notin \mathcal{R}(1/n_K)$, all manifolds constructed above are rational homology $3$--spheres. By the result of Ohta and Ono \cite{OO:weak-to-strong}, we can perturb the symplectic structure of any weak filling of a rational homology $3$--sphere to be a strong filling. Thus the resulting contact structures are strongly fillable, but not Liouville fillable.

  The proof is identical when $\phi$ is not pseudo-Anosov. The only difference is that the dividing slope of $N$ in $(Y,\xi_1)$ is $1 / \lfloor c(\phi) + 2 \rfloor$, so the surgery coefficients are in $\mathcal{R}(1/n_K)$ as desired.
\end{proof}

\begin{proof}[Proof of Theorem~\ref{thm:weak-non-Liouville}]
  The proof is identical to the proof of Theorem~\ref{thm:non-Liouville}, except that $Y_r(K)$ is not a rational homology sphere in general, so we cannot perturb a weak filling to a strong one. 
\end{proof}

\section{Contact structures on \texorpdfstring{$-\Sigma(2,3,6n\pm1)$}{-Sigma(2,3,6n+1)}}

In this section, we prove Theorem~\ref{thm:triangle}. To do so, we first review the construction of the tight contact structures on $-\Sigma(2,3,6n\pm1)$ by Ghiggini and Van-Horn-Morris \cite{GVHM:classification}, and by Tosun \cite{Tosun:classification}. We notice that for most of these contact structures, they used mixed Legendrian knots to construct them, so we can apply Theorem~\ref{thm:mixed}. 

\begin{proof}[Proof of Theorem~\ref{thm:triangle}]
Ghiggini and Van-Horn-Morris \cite{GVHM:classification} constructed tight contact structures on $-\Sigma(2,3,6n-1)$ as follows. Let $(M,\eta^i)$ be a rotative torus bundle where $M$ is obtained by $0$--surgery on the right-handed trefoil $T_{2,3}$. In $(M,\eta^i)$, they constructed Legendrian knots $L_{l,r}$ where $j = l-r$ and $n = l+r+i+2$. See Figure~3 of \cite{GVHM:classification}. They showed that Legendrian surgery on $L_{l,r}$ produces $\eta^n_{i,j}$. Notice that $L_{l,r}$ is a mixed Legendrian knot if $l,r > 0$.  Thus we have
\begin{gather*}
i = n - 2 - l -r < n - 3,\\
|j| = |l-r| < |l + r| = n-i-2.  
\end{gather*}  

Conversely, if $l = 0$ (the same argument works when $r=0$), then $|j| = n-i-2$. This implies that any $\eta^n_{i,j}$ with $i < n-3$ and $|j| < n-i-2$ are obtained from Legendrian surgery on a mixed knot. Since $\eta^i$ are not Liouville fillable for $i>0$ by Theorem~\ref{thm:torus-fillability}, we can apply Proposition~\ref{prop:mixed} and $\eta^n_{i,j}$ are not Liouville fillable. 

Tosun \cite{Tosun:classification} constructed tight contact structures on $-\Sigma(2,3,6n+1)$ as follows. Let $(M,\xi^i)$ be a rotative torus bundle where $M$ is obtained by $0$--surgery on the left-handed trefoil $T_{2,-3}$. In $(M,\xi^i)$, he constructed Legendrian knots $L_{l,r}$ where $j = l-r$ and $n = l+r+i+1$. See Figure~6 of \cite{Tosun:classification}. He showed that Legendrian surgery on $L_{l,r}$ produces $\xi^n_{i,j}$. Notice that $L_{l,r}$ is a mixed Legendrian knot if $l,r > 0$.  Thus we have
\begin{gather*}
i = n - 1 - l -r < n - 2,\\
|j| = |l-r| < |l + r| = n-i-1.  
\end{gather*}  

Conversely, if $l = 0$ (the same argument works when $r=0$), then $|j| = n-i-1$. This implies that any $\xi^n_{i,j}$ with $i < n-2$ and $|j| < n-i-1$ are obtained from Legendrian surgery on a mixed knot. Since $\xi^i$ are not Liouville fillable for $i>0$ by Theorem~\ref{thm:torus-fillability}, we can apply Proposition~\ref{prop:mixed} and $\xi^n_{i,j}$ are not Liouville fillable.
\end{proof}

\bibliography{references}
\bibliographystyle{plain}
\end{document}